\documentclass{article}
\usepackage{amsmath,amsfonts,amsthm,amscd,amssymb}
\usepackage{pstricks}
\usepackage{caption}
\usepackage{graphicx}
\usepackage{tabulary}
\newtheorem{thm}{Theorem}[section]
\newtheorem{pro}[thm]{Proposition}
\newtheorem{cor}[thm]{Corollary}
\newtheorem{lem}[thm]{Lemma}

\newtheorem{example}[thm]{Example}

\newtheorem{defn}[thm]{Definition}

\newcommand{\ceil}[1]{\left\lceil #1\right\rceil}

\begin{document}

\author{Mark Herman and Jonathan Pakianathan}
\title{A note on the unit distance problem for planar configurations with $\mathbb{Q}$-independent direction set}
\maketitle

\begin{abstract}
Let $T(n)$ denote the maximum number of unit distances that a set of $n$ points in the Euclidean plane $\mathbb{R}^2$ can determine with the additional
condition that the distinct unit length directions determined by the configuration must be $\mathbb{Q}$-independent.
This is related to the Erd\"os unit distance problem but with a simplifying additional assumption on the direction set which holds ``generically''.

We show that $T(n+1)-T(n)$ is the Hamming weight of $n$, i.e., the number of nonzero binary coefficients in the binary expansion of $n$, and find a formula for $T(n)$ explicitly. In particular $T(n)$  is $\Theta(n log(n))$. Furthermore we describe a process to construct a set of $n$ points in the plane with $\mathbb{Q}$-independent
unit length direction set which achieves exactly $T(n)$ unit distances.
In the process of doing this, we show $T(n)$ is also the same as the maximum number of edges a subset of vertices of size $n$ determines in either the
countably infinite lattice $\mathbb{Z}^{\infty}$ or the infinite hypercube graph $\{0,1\}^{\infty}$.

The problem of determining $T(n)$ can be viewed as either a type of packing or isoperimetric problem.


\end{abstract}


\section{Introduction}

Erd\"os posed the following question \cite{Erd46}: What is the maximum possible number $u(n)$ of unit distances determined by an $n$-point set in the Euclidean plane?
Erd\"os conjectured that $u(n)=O(n^{1+\epsilon})$ for any $\epsilon >0$ but the best that is currently known is that $u(n) = O(n^{\frac{4}{3}})$, see for example
\cite{SST84}, \cite{CEG90}, \cite{AS02} and \cite{Sze97}.

Examples are also known (see \cite{BMP05} and in fact some are also described in this paper) that show that $u(n) = \Omega(n \log n)$.
Any upper bound for the Erdos unit distance conjecture would give a lower bound on the number of distinct distances determined by
$n$ points in the Euclidean plane by simple pigeonholing. The hence related Erd\"os distinct distance conjecture, recently proven by Guth and Katz \cite{GK10}, says
that the number of distinct distances determined by $n$ points in the Euclidean plane is at least $\Omega(\frac{n}{\log(n)})$.

It follows from some recent work of J. Matou\u sek \cite{Mat11} that if a set of $n$ points, $n \geq 4$, in the plane determines at least $Cn \log(n) \log\log(n)$ unit distances then
there must be some integer dependencies within its unit distance set, i.e. some achieved unit difference vector is a $\mathbb{Z}$-linear combination of others. This indicates that configurations that are extremal for the Erd\"os unit distance problem
should contain lots of integral (and hence rational) dependencies in its unit distance set. In particular, it follows from Matou\u sek's work that if
a set of $n$ points has rationally independent distinct unit distances, then it determines no more than $O(n \log(n) \log\log(n))$ unit distances.

In this note, we will study these types of sets exclusively. More precisely note that if $A$ is a finite subset of the Euclidean plane $\mathbb{R}^2$ then
$A$ determines a unit vector set $U_{A} = \{ x-y | x,y \in A, \| x-y \|=1 \} \subset S^1$ where $S^1$ is the unit circle. Note that $U_{A}$ is symmetric i.e. if $u \in U_{A}$ so is $-u$.
Choose a representative of each direction represented in $U_{A}$ which
has principal argument (in radians) in the interval $[0,\pi)$ to form the unit direction set $D_{A}= \{ x-y | x,y \in A, \| x-y \|=1, Arg(x-y) \in [0,\pi) \} \subset S^1$.
Thus $U_A$ is the disjoint union of $D_A$ and $-D_A$.

We say that $A$ has rationally independent unit direction set if the elements of $D_A$ are independent over $\mathbb{Q}$, where $\mathbb{Q}$ is the field of
rational numbers. We show in this paper that most point configurations have rationally independent unit direction set and that this condition is generic in a sense.
Note this condition imposes no constraint on differences which are not of unit length nor does it preclude the possibility of a given unit direction vector occurring
multiple times in the difference set of $A$, it only imposes the condition of $\mathbb{Q}$-indepedence on the {\bf distinct} unit distances as determined in $D_A$.

We then ask the following question: What is the maximum possible number $T(n)$ of unit distances determined by an $n$-point set $A$ in the
Euclidean plane with rationally independent unit direction set?

In this paper we find an explicit exact formula for $T(n)$ and describe a process of constructing a configuration of $n$-points in the plane with rationally
independent unit direction set that achieves $T(n)$ unit distances.

\begin{thm}
Let $n \geq 1$ then $T(n)$, the maximum number of unit distances that a $n$-point set of the Euclidean plane with rationally independent unit direction set determines
is also equal to: \\
(a) The maximum number of edges determined by $n$ points in the standard countably infinite integer lattice $\mathbb{Z}^{\infty}$. \\
(b) The maximum number of edges determined by $n$ points in the standard countably infinite hypercube graph $\{0,1\}^{\infty}$. \\
Furthemore $T(n+1)-T(n)=H(n)$ where $H(n)$ is the Hamming weight of $n$ i.e., the number of nonzero binary coefficients in the binary expansion of
$n$. If $n=\sum_{j=1}^t 2^{k_j}$ where $k_1 > k_2 > \dots > k_t \geq 0$ then $T(n)=\sum_{j=1}^t (k_j 2^{k_j-1} + (j-1)2^{k_j})$ and in particular one always has
$$
\frac{n}{4}(\ceil{log n} -1) < T(n) < n \ceil{log n}
$$
where $log$ is the base 2 logarithm and $\ceil{x}$ is the smallest integer greater or equal to $x$. Thus $T(n)$ is big-Theta of $n log n$.
\end{thm}

The equivalence of (a) is proven in Theorem~\ref{thm: Zinf}, the equivalence of part (b) is proven in Corollary~\ref{cor: Zd max = unit cube max} and
the formulas for $T(n)$ are proven in Theorem~\ref{thm: count}.

As the first step, we show that any extremal configuration ($n$-point set in the Euclidean plane with $\mathbb{Q}$-independent unit direction set and
which determines the maximum $T(n)$ of unit distances) can be assumed to have a {\bf good} set of unit directions as defined below and furthermore
can be assumed to lie in the lattice generated by this good set of directions.

A set of unit directions $u_1, \dots u_{\ell} \in S^1$ is {\bf good} if $\| a_1 u_1 + \dots + a_{\ell} u_{\ell} \|=1$ for $(a_1, \dots, a_{\ell}) \in \mathbb{Z}^{\ell}$ if and only
if exactly one $a_j$ is nonzero. We also show that {\bf good} unit direction sets are ``generic'' (a dense $G_{\delta}$ set) in the space of collections of unit directions and in particular, that ``most''
$n$-point sets in $\mathbb{R}^2$ have $D_A$ a good unit direction set. Given this, we describe a process of making a $n$ point set with rationally
independent unit direction set which achieves the maximum possible number $T(n)$ of unit distances:

\begin{pro}
\label{pro: main}
Let $n \geq 1$ be given. Choose $d$ such that $2^{d-1} < n \leq 2^d$ and any collection of unit directions $u_0, \dots, u_{d-1} \in S^1$ which is {\bf good}.
Then let
$$
A = \{ a_0 u_0 + \dots + a_{d-1} u_{d-1} | a_j \in \{0, 1 \}, 0 \leq a_0+a_12 + \dots + a_{d-1}2^{d-1} < n \}
$$
Then $A$ is a $n$-point set in the Euclidean plane with rationally independent unit distance set which determines exactly the maximum possible $T(n)$ unit distances.
\end{pro}

 Proposition~\ref{pro: main} follows from the reductions in section~\ref{section: good directions} and
the results in section~\ref{sec: edgemax}. Using the notion of good directions sets, we reduce the problem to a packing problem in $\mathbb{Z}^{\infty}$, the countably infinite integer lattice and
then to a problem in $\{0,1\}^{\infty}$ the countably infinite hypercube graph. These lattice reductions are similar to what is discussed in work of
Brass (see \cite{BMP05} for a reference), but with a particular lattice arising specific to our particular situation. We reduce to this canonical lattice by
reducing the problem to configurations with good unit direction sets.

This final packing/graph theoretic result for hypercube graphs is interesting of its own right and had been previously studied. The result says that the best way to choose $n$ points, $n \leq 2^d$, in the
$d$-dimensional hypercube $\{0,1\}^d$ graph to maximize the number of edges determined (an edge is determined every time two vectors have Hamming distance one, i.e., differ in exactly one coordinate) is to choose the $n$ points as the binary representations of the numbers $0$ through $n-1$.

\begin{thm}
Fix $n \leq 2^d$. Let $V$ be a $n$-point subset of the hypercube graph on vertex set $\{0,1\}^d$ which determines $E(V)$ edges.
Let $V_n \subseteq \{0,1\}^d$ be the vertex set on $n$ vertices given by the binary expansions of the numbers $0$ through $n-1$,
then $E(V_n) \geq E(V)$.
\end{thm}

We provide an alternate proof of this result in this paper for completeness and because we feel it is somewhat shorter and cleaner than those that have previously appeared
in the literature \cite{Ber67}, \cite{Harp64}, \cite{Hart76}. These ``packing problems'' are also related to isoperimetric problems \cite{Bez94}.

\section{Good direction sets}
\label{section: good directions}

Let $A$ be an $n$-point set in the Euclidean plane ($\mathbb{R}^2$ equipped with the Euclidean metric). We define its unit difference set
$$
U_{A} = \{ x-y | x,y \in A, \| x-y \|=1 \} \subset S^1
$$ where $S^1$ is the unit circle. Note that $U_{A}$ is symmetric i.e. if $u \in U_{A}$ so is $-u$.
Choose a representative of each direction represented in $U_{A}$ which
has principal argument (in radians) in the interval $[0,\pi)$ to form the unit direction set
$$D_{A}= \{ x-y | x,y \in A, \| x-y \|=1, Arg(x-y) \in [0,\pi) \} \subset S^1.$$

We will say that $A$ has rationally independent unit direction set if the elements of $D_A$ are independent over $\mathbb{Q}$, where $\mathbb{Q}$ is the field of
rational numbers.

The unit direction graph determined by $A$ is the graph whose vertex set is the set $A$ and where there is an edge between $a_1, a_2 \in A$ if and only if
$\|a_1 - a_2\|=1$.

\begin{defn}
An {\bf extremal configuration} will denote an $n$-point set $A$ in the Euclidean plane with rationally independent unit direction set which achieves the maximum
possible number $T(n)$ of unit distances subject to these conditions.
\end{defn}

Note that trivially $T(1)=0$ so we might as well restrict out attention to $n \geq 2$. Our first lemma shows that we may always assume certain properties about our extremal configurations:

\begin{lem}
\label{lem: basicpicture}
For any fixed $n \geq 2$, there exists an extremal configuration $A$ such that $(0,0) \in A, (1,0) \in D_A$ and the unit direction graph determined by $A$ is
path connected. Thus $A \subset \mathbb{Z}(D_A) \cong \mathbb{Z}^d$ where $d=|D_A|$. Here $\mathbb{Z}(D_A)$ denotes the set of all $\mathbb{Z}$-linear combinations of elements in $D_A$. We also have $1 \leq d \leq n-1$.
\end{lem}
\begin{proof}
It is clear that as $n \geq 2$, $D_A$ is nonempty thus there is a pair of points $a_1, a_2 \in A$ such that $\|a_1-a_2\|=1$. After applying a translation and a rotation
to all the points of $A$ (which changes $D_A$ by a rotation so neither disturbs $\mathbb{Q}$-independence or $|D_A|$), we may assume $a_1=(0,0)$ and $a_2=(1,0)$ lie in $A$. Thus WLOG $(0,0) \in A$ and $(1,0) \in A$ and hence $(1,0) \in D_A$ also.

Now suppose that the unit direction graph of $A$ was disconnected. Let $C_1$ be a component and $C_2$ the rest of the graph. Then $C_1$ contains a
right-most vector (vector with maximal $x$-coordinate) $\hat{u}$ and $C_2$ contains a left-most vector (vector with minimal $x$-coordinate) $\hat{v}$.
We may then translate $C_2$ so that the image of $\hat{v}$ lies one unit to the right of $\hat{u}$. This changes no internal connections in either $C_1$ (which was left alone) or $C_2$ but raises the number of connections between them. It does not change the direction set $D_A$ either as $(1,0) \in D_A$ already and
no other unit directions were created in the process as the translated $C_2$ as a set lies one unit to the right of the rightmost point(s) of $C_1$. Thus we achieve a new extremal configuration with at least one more unit distance than the one we started with which contradicts the extremality of the original configuration. Thus an extremal configuration always has (path) connected unit direction graph.

Now as the unit direction graph of $A$ is path connected and $(0,0) \in A$, we know for any $x \in A$, it is possible to go from $(0,0)$ to $x$ with a path of edges
consisting of vectors in $U_A=D_A \cup -D_A$. Thus $A$ is contained in the $\mathbb{Z}$-span of $D_A$ as claimed. Finally as the elements of
$D_A$ are $\mathbb{Q}$-independent, they are also $\mathbb{Z}$-independent so the $\mathbb{Z}$-span of $D_A$ is a free abelian group of rank $d$
where $d=|D_A|$.

List $A$ as $A=\{a_1, \dots, a_n\}$. Now note that the $\mathbb{Z}$-span of $D_A$ is itself contained in the $\mathbb{Z}$-span of $\{ a_i - a_j | a_i, a_j \in A \}$. This later set is contained
in the $\mathbb{Z}$-span of the set $\{ a_2-a_1, a_3-a_1, \dots, a_n-a_1 \}$, a free abelian group of rank $\leq n-1$. Thus $d \leq n-1$ as the rank $d$ of a subgroup of
free abelian group of rank $n-1$ must have $d \leq n-1$ as $\mathbb{Z}$ is a PID.

\end{proof}

Note that while the $\mathbb{Z}$-span of $D_A$ is a subgroup of $\mathbb{R}^2$ isomorphic to $\mathbb{Z}^d$, it cannot be a closed subgroup unless
$d \leq 2$. This is because by Lie theory, the only closed subgroups of $\mathbb{R}^2$ are isomorphic to $\mathbb{R}^m, \mathbb{Z}^m$ with $m \leq 2$
or $\mathbb{R} \times \mathbb{Z}$, thus the closure of the $\mathbb{Z}$-span of $D_A$ must be all of $\mathbb{R}^2$ whenever $d > 2$, i.e., it must generate a dense ``lattice''.

By Lemma~\ref{lem: basicpicture}, an extremal configuration $A$ of any size $n \geq 2$ may be constructed by choosing a ``lattice'' $\mathbb{Z}^d \subseteq \mathbb{R}^2$ and distributing $n$ points in it to maximize the number of unit distances achieved by the set. Furthermore the basis vectors of this lattice have unit length and can be assumed to be achieved by the set as differences. Note that any other pair of points of unit distance in the lattice would have difference vector in the $\mathbb{Z}$-span of these basis vectors and so such a pair would have to be avoided by the set $A$ to maintain the condition of $\mathbb{Q}$-independence of unit directions.

This indicates that a ``good lattice'' to work with would be one where the only pairs of lattice points with unit distance apart are those which differ by plus or minus
a basis vector. This motivates the definition of a ``good direction set'' which we make next:

\begin{defn}
Let $S^1$ be the unit circle in the Euclidean plane and let $E \subset S^1$ be the set of points on the unit circle whose principal argument lies in
the interval $[0,\pi)$. Let $F(E,d) = \{ (u_1, \dots, u_d) | u_j \in E, u_i \neq u_j, 1 \leq i < j \leq d \}$ be the configuration space of $d$-tuples of distinct
directions in $E$, topologized as a subspace of $E^d$ with the product topology.

$(u_1,\dots,u_d) \in F(E,d)$ is a {\bf good set of directions} if whenever $\sum_{k=1}^d a_k u_k \in S^1$ with $(a_1,\dots,a_d) \in \mathbb{Z}^d$, we necessarily have
all but one $a_j$ is zero. (Note that this nonzero $a_j$ would then have to be $\pm 1$.)
\end{defn}

Note that $E^d$ is homeomorphic to $[0,\pi)^d$ and is hence locally compact, Hausdorff. $F(E,d)$ is an open subset of $E^d$ and so is also locally compact, Hausdorff and hence in particular a Baire space. Further note that a good set of directions has $\mathbb{Z}$-span equal to a ``lattice'' where the only pairs of vertices which have unit distance apart occur when the difference is plus or minus a basis vector. We record some properties of good direction sets in the next proposition.

\begin{pro}
Let $\{u_1, \dots, u_d\} \subset E \subset S^1$ be a good direction set, then $\{u_1, \dots, u_d\}$ is $\mathbb{Q}$-independent and thus the $\mathbb{Z}$-span of $\{u_1,\dots,u_d\}$
 is isomorphic to $\mathbb{Z}^d \subset \mathbb{R}^2$. We can then express any $\hat{v}$ in the $\mathbb{Z}$-span of $\{u_1,\dots,u_d\}$ uniquely as an integer $d$-tuple $(a_1,\dots,a_d)$ where $\hat{v}=\sum_{j=1}^d a_ju_j$. The $a_j$'s are referred to as the lattice coordinates of $\hat{v}$.
 Then in this lattice two vectors $\hat{v}_1,\hat{v}_2$ have Euclidean distance one if and only if they have
 $\ell_1$ distance one i.e., if and only if they differ in exactly one lattice coordinate by an amount $\pm 1$.
\end{pro}
\begin{proof}
Let $\{u_1, \dots, u_d\}$ be a good direction set. If $d=1$ the set is $\mathbb{Q}$-independent as $u_1 \in S^1$ is nonzero so WLOG $d \geq 2$.
Suppose that this set were rationally dependent. Then
$q_1 u_1 + \dots + q_d u_d = 0$ for rational numbers $q_j$, not all zero. We can then clear denominators and conclude that
$a_1 u_1 + \dots + a_d u_d = 0$ for integers $a_j$, not all zero. WLOG suppose $a_1 \neq 0$, then we can write
$(a_1-1)u_1 + \dots + a_du_d = -u_1$ and so the left hand side of the equation lies on $S^1$.  Since the set of directions is good, this implies
either $a_2=\dots=a_d=0$ which implies $a_1u_1=0$ so $a_1=0$ also a contradiction, or that $a_1=1$, and all but one of the other $a_j$'s are equal to $0$.
Without loss of generality let $a_2$ be the other $a_j$ that is nonzero then we get $u_1 + a_2u_2 = 0$ which implies $u_1 = -a_2u_2$ implying $a_2 = \pm 1$.
However both $u_1,u_2 \in E$ so we conclude $u_1=u_2$. This is a contradiction as good direction sets are defined to consist of distinct elements or
by noticing this would mean $1u_1 -2u_2 = -u_1 \in S^1$ contradicting the good direction condition. Thus we conclude good direction sets are $\mathbb{Q}$-independent and so the abelian group they
generate is free abelain of rank $d$ as claimed. Finally if $v_1 = a_1u_1 + \dots + a_du_d$ and $v_2=b_1u_1+\dots+b_du_d$ for $(a_1,\dots,a_d),(b_1,\dots,b_d) \in \mathbb{Z}^d$, it is easy to check that
$\|v_1-v_2\|=1$ if and only if $\|(a_1,\dots,a_d) - (b_1,\dots,b_d)\|_1=1$ using the definition of good direction set.
(Here $\| \cdot \|$ is the Euclidean norm on $\mathbb{R}^2$ while $\| \cdot \|_1$ is the $\ell_1$-norm on $\mathbb{Z}^d$.)

\end{proof}

\begin{example}
Identify $\mathbb{R}^2$ with the complex plane $\mathbb{C}$. The set $\{1, \xi=e^{\frac{2\pi i}{3}}\} \subset E \subset S^1$ is $\mathbb{Q}$-independent but is not a good
direction set as $1(1) + 1(\xi) = -\xi^2 \in S^1$.
\end{example}

We now show that the set $G \subset F(E,d)$ of good sets of unit directions is a dense $G_{\delta}$ set. Recall a $G_{\delta}$ set is a countable intersection of
open sets.

\begin{thm}
\label{thm: gooddir}
Let $G \subseteq F(E,d)$ be the set of good direction sets $(u_1,\dots, u_d)$ then $G$ is a dense $G_{\delta}$ set in $F(E,d)$.
\end{thm}
\begin{proof}
First note that when $d=1$, $G=F(E,1)$ and so there is nothing to show. Thus WLOG assume $d \geq 2$. For any $(a_1,\dots, a_d) \in \mathbb{Z}^d$ with
two or more $a_j$ nonzero, define $A_{(a_1,\dots,a_d)} = \{ (u_1,\dots,u_d) \in F(E,d) | \sum_{j=1}^d a_j u_j \in S^1 \}$.
Note that the function $f_{(a_1,\dots,a_d)} (u_1,\dots,u_d) = a_1u_1 + \dots + a_d u_d:  F(E,d) \to \mathbb{R}^2$ is continuous and so
$A_{(a_1,\dots,a_d)} = f_{(a_1,\dots,a_d)}^{-1}(S^1)$ is hence a closed subset of $F(E,d)$. We now show that $A_{(a_1,\dots,a_d)}$ is nowhere dense
in $F(E,d)$. Toward this let $U \subseteq F(E,d)$ where $U$ is nonempty and open. We have to show that there exists $(v_1,\dots,v_d) \in U$ such that
$(v_1,\dots,v_d) \notin A_{(a_1,\dots,a_d)}$.

Setting $\theta_{jk}$ to be the angle between $v_j$ and $v_k$, we have $(v_1,\dots,v_d) \in A_{(a_1,\dots,a_d)}$ if and only if
$\sum_{j < k} 2a_ja_k cos(\theta_{jk}) = 1 - \sum_{j=1}^d a_j^2$.

Now note if we parametrize $E \subset S^1$ by the principal argument and $t_j \in [0,\pi)$ is the argument of $v_j$, we have $\theta_{ij}=t_i-t_j$ and so the last equation can be written as:
$$
2\sum_{i < j }^d a_i a_j cos(t_i-t_j)= 1-a_1^2 - \dots - a_d^2.
$$
Suppose contrary to what we want to show that $U \subseteq A_{(a_1,\dots,a_d)}$ then the equation above holds for all choices of $t_1, \dots, t_d$
in an nonempty open subset of $F([0,\pi),d) \subseteq [0,\pi)^d$ and so we may differentiate it with respect to any of the variables $t_j$ on this open set. Differentiating this last identity with respect to $t_1$ twice would kill all terms not involving $t_1$ and replicate those involving $t_1$ yielding:
$$
\sum_{1 < j}^d a_1 a_j cos(t_1-t_j) = 0
$$
and thus imply that the net contribution of all terms involving $t_1$ in the original identity is zero! Continuing in this way with $t_2$ and then $t_3$ etc., we see that the identity can only hold for a nonempty open set of $(t_1,\dots, t_d)$ if and only if its left hand side is identically zero. This yields
$$
0 = 1 - a_1^2 - \dots - a_d^2
$$
which is a contradiction as at least two of the $a_j$'s are nonzero integers and so the right hand side is negative.

Thus $U \not\subset A_{(a_1,\dots,a_d})$ and we can conclude that $A_{(a_1,\dots,a_d)}$ is a nowhere dense closed set.

As $G=F(E,d) - \cup_{(a_1,\dots,a_d) \in \mathbb{Z}^d} A_{(a_1,\dots,a_d)}$ where the union is taken only over $d$-tuples in $\mathbb{Z}^d$ with at least two coordinates nonzero, we conclude $G$ is a dense $G_{\delta}$ set as its complement is a countable union of nowhere dense closed sets in the Baire space $F(E,d)$ (see \cite{Mu}).

\end{proof}

Let $\mathbb{Z}^d$ denote the integer lattice graph whose vertices are given by the set $\mathbb{Z}^d$ and where there is an edge joining
$(u_1,\dots,u_d)$ with $(v_1,\dots,v_d)$ if and only if they differ in exactly one coordinate and in this coordinate the entries differ by $\pm 1$.
Alternatively $\hat{u}$ and $\hat{v}$ are joined by an edge exactly when $\|\hat{u}-\hat{v}\|_1 = 1$ where $\| \cdot \|_1$ is the $\ell^1$-norm.

Note that we may regard these graphs as nested $\mathbb{Z}^1 \subset \mathbb{Z}^2 \subset \mathbb{Z}^3 \subset \dots$ and
we define $\mathbb{Z}^{\infty}$ as their union graph, the integer lattice graph on countably infinitely many coordinates. The vertices of this graph are eventually
zero integer sequences and two such sequences $\hat{u}, \hat{v}$ are joined by an edge if $\|\hat{u}-\hat{v}\|_1=1$.

The next theorem relates $T(n)$, the maximum number of unit distances determined by $n$ points in the Euclidean plane with $\mathbb{Q}$-independent
unit distance set, with the maximum number of edges that $n$ points in $\mathbb{Z}^{\infty}$ can determine.

\begin{thm}
\label{thm: Zinf}
$T(n)$ is equal to the maximum number of edges that $n$ points in $\mathbb{Z}^{\infty}$ can determine. It is also the maximum number of edges
that $n$ points in $\mathbb{Z}^{n-1}$ can determine.
\end{thm}
\begin{proof}
When $n=1$, $T(1)=0$ and there is nothing to show so assume $n \geq 2$. By Lemma~\ref{lem: basicpicture}, we may take $A$ an extremal set of $n$
points such that $A$ is a subset of the $\mathbb{Z}$-span of $D_A$ and hence is a subset of $\mathbb{Z}^d \subseteq \mathbb{Z}^{n-1}$.
However if $D_A$ is not a good set of directions, there can be some pairs $\hat{u}, \hat{v}$ in this ``lattice'' which have Euclidean unit distance apart but which have
$\|\hat{u} - \hat{v}\|_1 \neq 1$. For these pairs, we cannot have both members of the pair in $A$ as that would generate a unit difference vector (which is not one of the unit basis vectors of the lattice) which
is an integral combination of the basis unit vectors indicating a $\mathbb{Z}$-linear dependency in $D_A$. This contradicts that $A$ has rationally independent
unit direction vectors. Thus the placement of the $n$ points of $A$ in $\mathbb{Z}[D_A] \cong \mathbb{Z}^d$ must avoid placing two points of $A$ in these sorts of pair locations. Subject to these restrictions, $A$ is a placement of $n$ points in the standard $\mathbb{Z}^d$ graph that maximizes the number of edge connections.

On the other hand we know by Theorem~\ref{thm: gooddir}, that we can find $(u_1,\dots,u_d)$ a set of good directions that generate a ``good lattice''
$\mathbb{Z}^d$ and an arbitrary placement of $n$ points in this lattice will yield a set $B \subset \mathbb{R}^2$ with $D_B \subseteq \{u_1, \dots, u_d \}$ and hence rationally independent unit directions.

Comparing the situation for sets $A$ and $B$ we see that they are both subsets of the standard $\mathbb{Z}^d$ integer lattice graph but $A$ is constrained while $B$ is not, in the sense we may choose $B$ to correspond to $A$ under the isomorphism of these lattices or we could choose $B$ to be something else. Thus we see that we never lose if we choose $B$ inside a lattice generated by a good set of directions so as to maximize the number of edges
determined by $B$ within the graph $\mathbb{Z}^d$. Furthermore because the set of directions is good, the edge count within the standard $\mathbb{Z}^d$-lattice is the same as the unit distance count determined by the set $B$, i.e., two points in the lattice are a unit distance apart in the Euclidean metric if and only if
they are unit distance apart in the $\ell^1$-metric.

Thus extremal configurations for the Euclidean unit distance problem correspond to extremal configurations within the standard lattice $\mathbb{Z}^{\infty}$
or in fact
$\mathbb{Z}^d$ for some $d \leq n-1$.

Note also: given a set $C$ of $n$ points in the standard integer lattice $\mathbb{Z}^{\infty}$ that maximize edge count, we may translate one to the origin
and argue once again by extremality for path connectedness of the unit distance graph. From this it follows easily that there are at most $n-1$ coordinates for which the elements
of $C$ can have nonzero entries in these coordinates, i.e., $C$ can be viewed as a subset of $\mathbb{Z}^{n-1}$.

\end{proof}

\section{Edge maximizing configurations in $\mathbb{Z}^d$ and $\{0,1\}^d$}
\label{sec: edgemax}

 In the following, an induced subgraph on a set of vertices V refers to the subgraph within a given graph whose vertex set is $V$ and whose edge set is obtained by taking all the edges joining the vertices in $V$ in the ambient graph. Throughout this section let $\mathbb{Z}^d$ denote the standard $d$-dimensional integer lattice graph whose vertex set is $\mathbb{Z}^d$ and
where two integer vectors are adjacent if and only if they differ in exactly one coordinate by $\pm 1$. Furthermore let $\{0,1\}^l \times \mathbb{Z}^d \subset \mathbb{Z}^{d+l}$ and $\{0,1\}^d \subset \mathbb{Z}^d$ be given the induced subgraph structures also.

We seek to find edge maximizing configurations in the sense that we would like to place $n$ points in {\bf some} lattice $\mathbb{Z}^m$ so that the
number of edge connections amongst those points is maximized. The first lemma and its corollary shows that any such edge maximizing configuration can
be found inside the hypercube graph $\{0,1\}^d$ for some $d$.

Let $V \subseteq \mathbb{Z}^d$ be a finite set of vertices, we can and will always translate $V$ so that the number of edges it determines is unchanged and such that
$0 \in V$ and all $(x_1,\dots, x_d) \in V$ have $x_i \geq 0$ for all $i=1,\dots, d$. We then define $M_j(V) = \max_{v \in V}\{x_j\}$ and $M(V)=\max_{j \in \{1,\dots,d\}} M_j$.
Note that $M(V)$ is a non-negative integer.

\begin{lem}
\label{lem: Zd to Zd+1}
Let $V \subset \mathbb{Z}^d$ be a finite set and $G=(V,E)$ be the induced subgraph of $\mathbb{Z}^d$ and $M(V) \geq 2$ where $M(V)$
is the quantity defined in the previous paragraph. Then there exists $V' \subset \{0,1\}^l \times \mathbb{Z}^d$ and induced subgraph $G'=(V',E')$ of $\{0,1\}^l \times \mathbb{Z}^d$ with $|V|=|V'|$ and $|E|=|E'|$. Furthermore $M(V') = M(V)-1$.
\end{lem}
\begin{proof}
Let $M$ denote $M(V)$ throughout. Without loss of generality we assume $M_1(V)=M \geq 2$. Let $V_1$ be the set of vertices in $V$ with $x_1=M$ and let $V_0=V\setminus V_1$. Note that $v=(M,a_2,\cdots,a_d)\in V_1$ shares an edge with $w=(b_1,b_2,\cdots,b_d)\in V_0$ if and only if $b_1=M-1$ and $b_j=a_j$ for $j=2,\cdots,d$.

We map the elements of $V$ into $\{0,1\} \times \mathbb{Z}^d$ as follows: \\
if $v=(M,a_2,\cdots,a_d)\in V_1$, then $v\rightarrow (1,M-1,a_2,\cdots,a_d)$ and if $w=(b_1,\cdots,b_d) \in V_0$, then $w\rightarrow (0,b_1,\cdots,b_d)$. Let $G'=(V',E')$ be the graph resulting from this mapping. By construction $|V'|=|V|$. We see that $v\in V_1$ and $w\in V_0$ share in edge in $G$ if and only if their images share an edge in $G'$. Moreover, we also see that $v_1, v_2 \in V_1$ share an edge in $G$ if and only if their images share an edge in $G'$ and likewise $w_1, w_2 \in V_0$ share an edge in $G$ if and only if their images share an edge in $G'$. So $|E'|=|E|$. We have $M_2(V')=M_1(V)-1$ and $M(V)-1\leq M(V') \leq M(V)$. At this point, if $M(V')=M(V)-1$ we are done. Otherwise, there is another coordinate $x_j$ with $M_j(V')=M(V)$ and we repeat the process for this coordinate. Continuing in this fashion eventually results in a graph with $M(V')=M(V)-1$ since there are at most $d$ possible (original) coordinates to consider.

\end{proof}

\begin{cor}
\label{cor: Zd to unit cube}
Let $G=(V,E)$ be a finite induced subgraph of the standard integer lattice graph $\mathbb{Z}^d$. Then there exists an induced subgraph $G'=(V',E')$ on $\{0,1\}^{d'}$ for some $d'\in \mathbb{N}$, with $|V|=|V'|$ and $|E|=|E'|$.
\end{cor}
\begin{proof}
After $M(V)-1$ iterations of lemma~\ref{lem: Zd to Zd+1} we have $M(V')=1$ which implies the resulting graph is now an induced subgraph of $\{0,1\}^{d'}$.
\end{proof}

In particular, suppose $K$ is an extremal graph on $\mathbb{Z}^\infty$, i.e., $|V(K)|=n$ and $|E(K)|$ equals the maximum number of edges determined by $n$ points in $\mathbb{Z}^{\infty}$. The corollary tells us that there is a graph $K'$ on $\{0,1\}^\infty$ with same number of vertices and edges, hence
$$
\max_{\substack{G\in \mathbb{Z}^\infty \\ |V(G)|=n}} |E(G)| =|E(K)|=|E(K')|\leq \max_{\substack{G\in \{0,1\}^\infty \\ |V(G)|=n}} |E(G)|
$$
Since the inequality in the other direction is obvious we have established the following.

\begin{cor}
\label{cor: Zd max = unit cube max}
The maximum number of edges determined by $n$ points in the standard countably infinite integer lattice $\mathbb{Z}^{\infty}$ is equal to the maximum number of edges determined by $n$ points in the standard countably infinite hypercube graph $\{0,1\}^{\infty}$.
\end{cor}

This reduces the problem of finding $T(n)$ to the edge maximizing problem on $\{0,1\}^{\infty}$ which we now focus on. The solution which we give in theorem \ref{thm: 2 cube} is known, and was proved in \cite{Ber67}, \cite{Harp64}, \cite{Hart76}. We provide an alternate proof of this result for completeness and because we feel it is somewhat shorter and cleaner than those that have previously appeared. See also \cite{Bez94} for an overview of related isoperimetric problems.

We label the elements of the $d$ dimensional unit hypercube $\{0,1\}^d$ via binary representation. That is, let $v_j=(a_{d-1},a_{d-2},\dots,a_0)$ where $j=\sum_{k=0}^{d-1} a_k 2^k$, each $a_k \in \{0,1\}$. In the following, for any graph embedded in $\{0,1\}^d$ it is implied that vertices $v_i$ and $v_j$ are adjacent if and only if the Hamming distance between $v_i$ and $v_j$ is 1, i.e., if and only if $v_i$ and $v_j$ differ in exactly one coordinate.

For the sake of clarity, we use the following example to introduce some definitions we will use. Although they are introduced in this example, it should be clear how they are defined in general.
\subsection{Example: $d=2^5$}
\label{sec: example}
Consider the array for $(x_4,x_3,\cdots,x_0)\in \{0,1\}^5$:
$$
\begin{array}{cc|cc}
\begin{array}{c} \text{\Large{Block}}\\ \text{\Large{A}} \end{array} & \begin{array}{c} v_0=(0,0,0,0,0) \\ v_1= (0,0,0,0,1) \\ v_2=(0,0,0,1,0) \\ v_3=(0,0,0,1,1)\\ \vdots \\ v_7=(0,0,1,1,1)\end{array} & \begin{array}{c} v_8=(0,1,0,0,0) \\ v_9=(0,1,0,0,1) \\ v_{10}=(0,1,0,1,0) \\ v_{11}=(0,1,0,1,1)\\ \vdots \\ v_{15}=(0,1,1,1,1) \end{array} & \begin{array}{c} \text{\Large{Block}}\\ \text{\Large{B}} \end{array}
\\\hline
\begin{array}{c} \text{\Large{Block}}\\ \text{\Large{C}} \end{array} & \begin{array}{c} v_{16}=(1,0,0,0,0) \\ v_{17}=(1,0,0,0,1) \\ v_{18}=(1,0,0,1,0) \\ v_{19}=(1,0,0,1,1)\\ \vdots \\ v_{23}=(1,0,1,1,1)\end{array} & \begin{array}{c} v_{24}=(1,1,0,0,0) \\ v_{25}=(1,1,0,0,1) \\ v_{26}=(1,1,0,1,0) \\ v_{27}=(1,1,0,1,1)\\ \vdots \\ v_{31}=(1,1,1,1,1) \end{array}& \begin{array}{c} \text{\Large{Block}}\\ \text{\Large{D}} \end{array}
\end{array}
$$
The blocks are determined by using the leftmost two coordinates: Block A: $(0,0,*)$,
Block B: $(0,1,*)$, Block C: $(1,0,*)$, Block D: $(1,1,*)$.

Arrays of this form will also be used to represent induced subgraphs of the $5$ dimensional unit hypercube by putting a dark dot next to vertices
that occur in the subgraph.  For instance, the induced subgraph $T$ with vertex set
$$
V=\{v_0,v_1,v_9,v_{10},v_{11},v_{19},v_{23},v_{24},v_{25},v_{26},v_{27},v_{28},v_{29},v_{30}\}
$$ is represented by
$$
\begin{array}{cc|cc}
\begin{array}{c} \text{\Large{Block}}\\ \text{\Large{A}} \end{array} & \begin{array}{l} v_0=(0,0,0,0,0) \ \bullet \\ v_1=(0,0,0,0,1) \ \bullet\\ v_2=(0,0,0,1,0)\\ v_3=(0,0,0,1,1)\\ v_4=(0,0,1,0,0)\\ v_5=(0,0,1,0,1)\\ v_6=(0,0,1,1,0) \\ v_7=(0,0,1,1,1) \end{array} & \begin{array}{l} v_8=(0,1,0,0,0) \\ v_9=(0,1,0,0,1) \ \bullet \\ v_{10}=(0,1,0,1,0) \ \bullet \\ v_{11}=(0,1,0,1,1) \ \bullet \\ v_{12}=(0,1,1,0,0) \\ v_{13}=(0,1,1,0,1) \\ v_{14}=(0,1,1,1,0) \\ v_{15}=(0,1,1,1,1) \end{array} & \begin{array}{c} \text{\Large{Block}}\\ \text{\Large{B}} \end{array}
\\\hline
\begin{array}{c} \text{\Large{Block}}\\ \text{\Large{C}} \end{array} & \begin{array}{l} v_{16}=(1,0,0,0,0) \\ v_{17}=(1,0,0,0,1) \\ v_{18}=(1,0,0,1,0) \\v_{19}=(1,0,0,1,1) \ \bullet \\ v_{20}=(1,0,1,0,0) \\ v_{21}=(1,0,1,0,1) \\ v_{22}=(1,0,1,1,0) \\ v_{23}=(1,0,1,1,1) \bullet \end{array} & \begin{array}{l} v_{24}=(1,1,0,0,0) \ \bullet\\ v_{25}=(1,1,0,0,1)\ \bullet \\ v_{26}=(1,1,0,1,0) \ \bullet \\  v_{27}=(1,1,0,1,1) \ \bullet \\ v_{28}=(1,1,1,0,0) \ \bullet \\ v_{29}=(1,1,1,0,1) \ \bullet \\ v_{30}=(1,1,1,1,0) \ \bullet \\ v_{31}=(1,1,1,1,1) \end{array}& \begin{array}{c} \text{\Large{Block}}\\ \text{\Large{D}} \end{array}
\end{array}
$$
In a graph $G$, let $V_A(G)$ denote the vertices of $G$ in block A ($V_B(G), V_C(G), V_D(G)$ defined analogously). Note that each vertex in $V_D(G)$ shares an edge with at most one vertex in $V_C(G)$ (the element directly to the left). For instance, in the graph above $v_{19}$ and $v_{27}$ share an edge but there are no other vertices in $V_D$ sharing an edge with a vertex in $V_C$. We call these ``horizontal'' edges, as with edges connecting a vertex of $V_B$ with the vertex to the left in $V_A$. The horizontal edges in a graph $G$ will be denoted $E^{hor}(G)$. In the graph above, $|E^{hor}|=2$, the horizontal edges are $[v_{19},v_{27}]$ and $[v_1,v_9]$. Similarly, each vertex of $V_D(G)$ shares an edge with at most one vertex of $V_B(G)$. For instance, in the graph above $v_{25}$ shares an edge with $v_{9}$ but no other vertices in $V_B$. We call these ``vertical'' edges as with edges connecting a vertex of $V_C(G)$ with a vertex of $V_A(G)$. The vertical edges in a graph $G$ will be denoted $E^{vert}(G)$. In the graph above, $|E^{vert}|=3$, the vertical edges are $[v_{9},v_{25}], [v_{10},v_{26}], [v_{11},v_{27}]$. Let $V_{AB}(G)=V_A(G)\cup V_B(G)$ and define $V_{AC}(G), V_{BD}(G)$ and $V_{CD}(G)$ analogously. One can observe that the following inequalities hold in general:
\begin{eqnarray}
|E^{hor}(G)| &\leq& \min\{|V_{AC}(G)|, |V_{BD}(G)|\} \label{hor}
\\
|E^{vert}(G)| &\leq& \min\{|V_{AB}(G)|, |V_{CD}(G)|\} \label{vert}
\end{eqnarray}

Let $E_A(G)$ be the edges in the subgraph of $G$ induced by $V_A(G)$, and $E_B(G), E_C(G), E_D(G)$ defined analogously. Similarly, let $E_{AB}(G)$ be the edges in the subgraph of $G$ induced by $V_{AB}(G)$, and $E_{AC}(G), E_{BD}(G), E_{CD}(G)$ defined analogously. Note that the edge set $E(G)$ of a graph G can be written as a disjoint union as
\begin{eqnarray}
E(G) &=& E_{AB}(G) \cup E_{CD}(G) \cup E^{vert}(G) \label{union1}
\\
E(G) &=& E_{AC}(G) \cup E_{BD}(G) \cup E^{hor}(G) \label{union2}
\end{eqnarray}

We also make the following definitions. We say a graph is ``completely arranged'' if $V=\{v_0,v_1,\cdots,v_I\}$ for some $I$. We say a graph is ``horizontally arranged'' if $V_{AB}=\{v_0,v_1,\cdots,v_J\}$ and $V_{CD}=\{v_{16},v_{17},\cdots,v_{K}\}$ for some $J$ and $K$. We say a graph is ``vertically arranged'' if $V_{AC}$ is filled consecutively downward starting from $v_0$, and $V_{BD}$ is filled consecutively downward starting from $v_{8}$. Completely arranging a graph is replacing it by the completely arranged graph with the same number of vertices. Horizontally arranging a graph is replacing it by the horizontally arranged graph with the same $|V_{AB}|$ and $|V_{CD}|$. Vertically arranging a graph is replacing it by the vertically arranged graph with the same $|V_{AC}|$ and $|V_{BD}|$. For example, the graph $T$ above is neither vertically nor horizontally arranged. Horizontally arranging $T$ gives us
$$
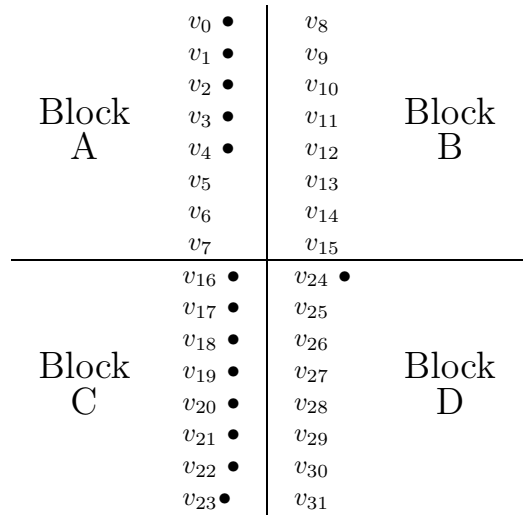
\begin{array}{cc|cc}
\begin{array}{c} \text{\Large{Block}}\\ \text{\Large{A}} \end{array} & \begin{array}{l} v_0 \ \bullet \\ v_1 \ \bullet\\ v_2 \ \bullet\\ v_3\ \bullet \\ v_4 \ \bullet\\ v_5\\ v_6 \\ v_7 \end{array} & \begin{array}{l} v_8 \\ v_9 \\ v_{10} \\ v_{11} \\ v_{12} \\ v_{13} \\ v_{14} \\ v_{15} \end{array} & \begin{array}{c} \text{\Large{Block}}\\ \text{\Large{B}} \end{array}
\\\hline
\begin{array}{c} \text{\Large{Block}}\\ \text{\Large{C}} \end{array} & \begin{array}{l} v_{16}\ \bullet \\ v_{17}\ \bullet \\ v_{18}\ \bullet \\v_{19} \ \bullet \\ v_{20}\ \bullet \\ v_{21}\ \bullet \\ v_{22}\ \bullet \\ v_{23} \bullet \end{array} & \begin{array}{l} v_{24} \ \bullet\\ v_{25} \\ v_{26} \\  v_{27} \\ v_{28} \\ v_{29} \\ v_{30} \\ v_{31} \end{array}& \begin{array}{c} \text{\Large{Block}}\\ \text{\Large{D}} \end{array}
\end{array}
$$
vertically arranging $T$ gives us
$$
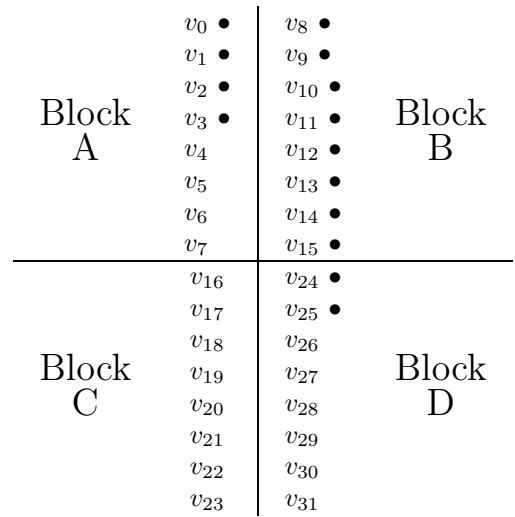
\begin{array}{cc|cc}
\begin{array}{c} \text{\Large{Block}}\\ \text{\Large{A}} \end{array} & \begin{array}{l} v_0 \ \bullet \\ v_1 \ \bullet\\ v_2 \ \bullet\\ v_3\ \bullet \\ v_4\\ v_5\\ v_6 \\ v_7 \end{array} & \begin{array}{l} v_8\ \bullet \\ v_9\ \bullet \\ v_{10}\ \bullet \\ v_{11}\ \bullet \\ v_{12}\ \bullet \\ v_{13}\ \bullet \\ v_{14}\ \bullet \\ v_{15}\ \bullet \end{array} & \begin{array}{c} \text{\Large{Block}}\\ \text{\Large{B}} \end{array}
\\\hline
\begin{array}{c} \text{\Large{Block}}\\ \text{\Large{C}} \end{array} & \begin{array}{l} v_{16} \\ v_{17}\ \\ v_{18} \\v_{19}\\ v_{20} \\ v_{21} \\ v_{22}\\ v_{23} \end{array} & \begin{array}{l} v_{24} \ \bullet\\ v_{25} \ \bullet \\ v_{26} \\  v_{27} \\ v_{28} \\ v_{29} \\ v_{30} \\ v_{31} \end{array}& \begin{array}{c} \text{\Large{Block}}\\ \text{\Large{D}} \end{array}
\end{array}
$$
and completely arranging $T$ gives us
$$
\begin{array}{cc|cc}
\begin{array}{c} \text{\Large{Block}}\\ \text{\Large{A}} \end{array} & \begin{array}{l} v_0 \ \bullet \\ v_1 \ \bullet\\ v_2 \ \bullet\\ v_3\ \bullet \\ v_4 \ \bullet\\ v_5\ \bullet\\ v_6\ \bullet \\ v_7\ \bullet \end{array} & \begin{array}{l} v_8\ \bullet \\ v_9\ \bullet \\ v_{10}\ \bullet \\ v_{11}\ \bullet \\ v_{12}\ \bullet \\ v_{13}\ \bullet \\ v_{14} \\ v_{15} \end{array} & \begin{array}{c} \text{\Large{Block}}\\ \text{\Large{B}} \end{array}
\\\hline
\begin{array}{c} \text{\Large{Block}}\\ \text{\Large{C}} \end{array} & \begin{array}{l} v_{16} \\ v_{17}\ \\ v_{18} \\v_{19}\\ v_{20} \\ v_{21} \\ v_{22}\\ v_{23} \end{array} & \begin{array}{l} v_{24} \\ v_{25} \\ v_{26} \\  v_{27} \\ v_{28} \\ v_{29} \\ v_{30} \\ v_{31} \end{array}& \begin{array}{c} \text{\Large{Block}}\\ \text{\Large{D}} \end{array}
\end{array}
$$

Let $\lambda_i$ be the graph automorphism of the hypercube graph $\{0,1\}^d$ that changes the $x_i$ coordinate of each vertex from 0 to 1 or 1 to 0 (for instance $(x_{d-1},x_{d-2},\cdots,x_i=1,\cdots,x_0)$ maps to $(x_{d-1},x_{d-2},\cdots,x_i=0,\cdots,x_0)$ and vice versa). Observe that applying $\lambda_4$ in this example interchanges block A with C and block B with D. In particular, $\lambda_4(T)$ (which changes the leftmost coordinate in the graph $T$) is given by
$$
\begin{array}{cc|cc}
\begin{array}{c} \text{\Large{Block}}\\ \text{\Large{A}} \end{array} & \begin{array}{l} v_0=(0,0,0,0,0) \\ v_1=(0,0,0,0,1) \\ v_2=(0,0,0,1,0)\\ v_3=(0,0,0,1,1) \ \bullet\\ v_4=(0,0,1,0,0)\\ v_5=(0,0,1,0,1)\\ v_6=(0,0,1,1,0) \\ v_7=(0,0,1,1,1) \ \bullet \end{array} & \begin{array}{l} v_8=(0,1,0,0,0) \ \bullet\\ v_9=(0,1,0,0,1) \ \bullet \\ v_{10}=(0,1,0,1,0) \ \bullet \\ v_{11}=(0,1,0,1,1) \ \bullet \\ v_{12}=(0,1,1,0,0) \ \bullet\\ v_{13}=(0,1,1,0,1) \ \bullet\\ v_{14}=(0,1,1,1,0)\ \bullet \\ v_{15}=(0,1,1,1,1) \end{array} & \begin{array}{c} \text{\Large{Block}}\\ \text{\Large{B}} \end{array}
\\\hline
\begin{array}{c} \text{\Large{Block}}\\ \text{\Large{C}} \end{array} & \begin{array}{l} v_{16}=(1,0,0,0,0) \ \bullet \\ v_{17}=(1,0,0,0,1) \ \bullet \\ v_{18}=(1,0,0,1,0) \\v_{19}=(1,0,0,1,1) \\ v_{20}=(1,0,1,0,0) \\ v_{21}=(1,0,1,0,1) \\ v_{22}=(1,0,1,1,0) \\ v_{23}=(1,0,1,1,1) \end{array} & \begin{array}{l} v_{24}=(1,1,0,0,0) \\ v_{25}=(1,1,0,0,1)\ \bullet \\ v_{26}=(1,1,0,1,0) \ \bullet \\  v_{27}=(1,1,0,1,1) \ \bullet \\ v_{28}=(1,1,1,0,0) \\ v_{29}=(1,1,1,0,1) \\ v_{30}=(1,1,1,1,0) \\ v_{31}=(1,1,1,1,1) \end{array}& \begin{array}{c} \text{\Large{Block}}\\ \text{\Large{D}} \end{array}
\end{array}
$$
Similarly, $\lambda_3(T)$ interchanges block A with B and C with D.

Let $\sigma_{i,j}$ be the graph automorphism of $\{0,1\}^d$ that interchanges the $x_i$ and $x_j$ coordinates of each vertex. Note that $\sigma_{3,4}(T)$ results in interchanging blocks B and C while blocks A and D are left alone.

We will utilize the facts that for any induced subgraph $G$, $|E(\lambda_i(G))|=|E(G)|$ and $|E(\sigma_{i,j}(G))|=|E(G)|$ in general as $\sigma_{i,j}$ and $\lambda_i$ are graph automorphisms
of the hypercube graph.

We are now ready to state and prove our result.

\begin{thm}
\label{thm: 2 cube}
Let $G$ be an induced subgraph of $\{0,1\}^d$ with vertices adjacent if and only if they are separated by a Hamming distance 1. Then $|E(G)|\leq |E(H)|$, where $H$ is the totally arranged graph with $|V(H)|=|V(G)|$, i.e., $V(H)=\{v_0,v_1,\cdots,v_{|V(G)|-1}\}$ where $v_j$ is defined by binary representation as above.
\end{thm}
\begin{proof}
Consider the array for $(x_{d-1},x_{d-2},\cdots,x_0)\in \{0,1\}^d$, there are $2^{d-2}$ vertices in each block:
$$
\begin{array}{cc|cc}
\begin{array}{c} \text{\Large{Block}}\\ \text{\Large{A}} \end{array} & \begin{array}{l} v_0 \ \\ v_1 \\ \vdots \\ v_{2^{d-2}-1} \end{array} & \begin{array}{l} v_{2^{d-2}} \\ v_{2^{d-2}+1} \\ \vdots \\ v_{2^{d-1}-1} \end{array} & \begin{array}{c} \text{\Large{Block}}\\ \text{\Large{B}} \end{array}
\\\hline
\begin{array}{c} \text{\Large{Block}}\\ \text{\Large{C}} \end{array} & \begin{array}{l} v_{2^{d-1}} \ \\ v_{2^{d-1}+1} \\ \vdots \\ v_{2^{d-1}+2^{d-2}-1} \end{array} & \begin{array}{l} v_{2^{d-1}+2^{d-2}} \\ v_{2^{d-1}+2^{d-2}+1} \\ \vdots \\ v_{2^{d}-1} \end{array}& \begin{array}{c} \text{\Large{Block}}\\ \text{\Large{D}} \end{array}
\end{array}
$$
We proceed by induction on $d$, applying a sequence of operations that does not decrease the number of edges in the graph and results in the totally arranged graph with the same number of vertices. Fix $d\geq 2$, let $G$ be an induced subgraph of $\{0,1\}^d$, and assume the conclusion holds for $d-1$. The inductive hypothesis tells us that horizontally arranging $G$ does not decrease $|E_{AB}|$ nor $|E_{CD}|$. This is because each horizontal row is isomorphic as a
graph to $\{0,1\}^{d-1}$ to which we can apply the inductive hypothesis. Since we have equality in equation (\ref{vert}) in a horizontally arranged graph, it follows from equation (\ref{union1}) that horizontally arranging $G$ does not decrease the total number of edges. Similarly, the inductive hypothesis tells us that vertically arranging $G$ does not decrease $|E_{AC}|$ nor $|E_{BD}|$ and from equations (\ref{hor}) and (\ref{union2}) we see that vertically arranging $G$ does not decrease the total number of edges.

We first vertically arrange $G$. If $|V_{AC}|<|V_{BD}|$, apply $\lambda_{d-2}$ (interchanging block A with B and block C with D). This results in a vertically arranged graph $G'$ with $|V_{AC}|\geq |V_{BD}|$, i.e., of the form:
$$
\begin{array}{cc|cc}
\begin{array}{c} \text{\Large{Block}}\\ \text{\Large{A}} \end{array} & \begin{array}{l} v_0 \ \bullet \\ v_1\ \bullet \\ \vdots \\ v_{2^{d-2}-1} \ \bullet \end{array} & \begin{array}{l} v_{2^{d-2}}\ \bullet \\ v_{2^{d-2}+1}\ \bullet \\ v_{2^{d-2}+2}\ \bullet \\ \vdots \\ v_{2^{d-1}-1}\ \bullet \end{array} & \begin{array}{c} \text{\Large{Block}}\\ \text{\Large{B}} \end{array}
\\\hline
\begin{array}{c} \text{\Large{Block}}\\ \text{\Large{C}} \end{array} & \begin{array}{l} v_{2^{d-1}}\ \bullet \\ v_{2^{d-1}+1}\ \bullet \\ v_{2^{d-1}+2}\bullet \\ \vdots \\ v_{2^{d-1}+2^{d-2}-1} \end{array} & \begin{array}{l} v_{2^{d-1}+2^{d-2}}\ \bullet \\ v_{2^{d-1}+2^{d-2}+1}\ \bullet \\ v_{2^{d-1}+2^{d-2}+2} \\ \vdots \\ v_{2^{d}-1} \end{array}& \begin{array}{c} \text{\Large{Block}}\\ \text{\Large{D}} \end{array}
\end{array}
$$
We now proceed by cases on the form of $G'$.

\begin{enumerate}

\item[Case 1:] Block B is full\\
By the form of $G'$, we know block A is full. We can then horizontally arrange, which results in a totally arranged graph.

\item[Case 2:] Block B is not full\\
By the form of $G'$, we know block $D$ is empty.\\
\begin{enumerate}

\item[2(a):] Block C is empty\\
Horizontally arranging results in a totally arranged graph.

\item[2(b):] Block C is non-empty\\
By the form of $G'$ we know block A is full. If $|V_B|<|V_C|$, we apply $\sigma_{d-2,d-1}$ (interchanging blocks B and C). If this results in B full or C empty the graph is totally arranged (since A is full and D is empty here). Otherwise, we now have the vertically arranged form:
$$
\begin{array}{cc|cc}
\begin{array}{c} \text{\Large{Block}}\\ \text{\Large{A}} \end{array} & \begin{array}{l} \text{\Large{FULL}} \end{array} & \begin{array}{l} \text{\Large{partially filled}} \end{array} & \begin{array}{c} \text{\Large{Block}}\\ \text{\Large{B}} \end{array}
\\[5mm]\hline
\begin{array}{c} \\ \text{\Large{Block}}\\ \text{\Large{C}} \end{array} & \begin{array}{l} \\ \text{\Large{partially filled}} \\ \text{\Large{with }} |V_C|\leq |V_B| \end{array} & \begin{array}{l} \\ \text{\Large{EMPTY}} \end{array}& \begin{array}{c} \\ \text{\Large{Block}}\\ \text{\Large{D}} \end{array}
\end{array}
$$
From equation (\ref{union1}) we see that $|E|=|E_{AB}|+|E_{CD}|+|E^{vert}|=|E_{AB}|+|E_{CD}|+|V_C|$ in this form. We now interchange blocks $C$ and $D$ giving us a graph of the form:
$$
\begin{array}{cc|cc}
\begin{array}{c} \text{\Large{Block}}\\ \text{\Large{A}} \end{array} & \begin{array}{l} \text{\Large{FULL}} \end{array} & \begin{array}{l} \text{\Large{partially filled}} \end{array} & \begin{array}{c} \text{\Large{Block}}\\ \text{\Large{B}} \end{array}
\\[5mm] \hline
\begin{array}{c} \\ \text{\Large{Block}}\\ \text{\Large{C}} \end{array} & \begin{array}{l} \\ \text{\Large{EMPTY}} \end{array} & \begin{array}{l} \\ \text{\Large{partially filled}} \\ \text{\Large{with }} |V_D|\leq |V_B| \end{array} & \begin{array}{c} \\ \text{\Large{Block}}\\ \text{\Large{D}} \end{array}
\end{array}
$$
Clearly this operation did not change $|E_{AB}|$ and $|E_{CD}|$ and observe that we now have $|E|=|E_{AB}|+|E_{CD}|+|E^{vert}|=|E_{AB}|+|E_{CD}|+|V_D|$ since $|V_D|\leq |V_B|$. Thus we see that the number of edges was preserved by this operation.

We now vertically arrange the graph. If this does not completely fill block B, we have both blocks C and D empty and the graph is totally arranged. If this does completely fill block B we can horizontally arrange which results in a totally arranged graph.

\end{enumerate}
\end{enumerate}

We see that in all cases, the graph can be totally arranged without reducing the number of edges.
\end{proof}

\begin{cor}
\label{cor: Zd to optimal cube}
Let $G=(V,E)$ be a finite induced subgraph of some standard integer lattice $\mathbb{Z}^{m}$ with $|V|=n$. Let $2^{d-1}<n\leq 2^{d}$. Then $|E(G)|\leq |E(H)|$, where $H$ is the totally arranged graph on $\{0,1\}^{d}$ with $|V(H)|=n$, i.e., $V(H)=\{v_0,v_1,\cdots,v_{n-1}\}$ where $v_j$ is defined by binary representation as above.
\end{cor}
\begin{proof}
This follows from corollary \ref{cor: Zd to unit cube} and theorem \ref{thm: 2 cube}.
\end{proof}

So the totally arranged graphs are extremal configurations which maximize edge count in $\{0,1\}^{\infty}$ and $\mathbb{Z}^{\infty}$ for a given number
of vertices. Thus counting their edges gives us $T(n)$. For a natural number $j$, let $H(j)$ be the number of non-zero digits in the binary expansion of $j$ (Hamming weight) and make the following observation: starting with the graph with vertex set $V=\{v_0,v_1,\dots,v_{k-1}\}$, if the single vertex $v_k$ is added the number of edges increases by $H(k)$, i.e., $\nabla T(k)=T(k+1)-T(k)=H(k)$. To see this, note that if we replace a single 1 in $v_k$ by 0 we obtain $v_j\in V$, whereas if we replace a single 0 in $v_k$ by 1 we obtain $v_j\notin V$. We then have $T(n)=T(n)-T(1)=\sum_{k=1}^{n-1} \nabla T(k)=\sum_{k=0}^{n-1} H(k)$. We will use this to obtain an exact expression for $T(n)$.

\begin{thm}
\label{thm: count}
(a) $T(2^d)=d 2^{d-1}$ for $d\in \mathbb{N}$.\\[2mm]
(b) More generally, if $n=\sum_{j=1}^t 2^{k_j}$, where $k_1>k_2>\cdots>k_t\geq 0$, then $T(n)=\sum_{j=1}^t (k_j 2^{k_j-1} +(j-1) 2^{k_j})$.\\[2mm]
(c) For all $n\in \mathbb{N}$, $n(\ceil{\log n}-1)/4 < T(n) < n\ceil{\log n}$, where $\log$ is the base 2 logarithm.
\end{thm}
\begin{proof}

(a) If $n=2^d$, we note that each vertex in the $d-$dimensional unit hypercube has degree $d$, so the sum of degrees over all vertices is $d 2^d$, then $T(n)=d 2^{d-1}$ follows from the handshaking theorem.\\

(b) For $2^{d-1} < n \leq 2^d$ and $n=\sum_{j=1}^t 2^{k_j}$, where $k_1>k_2>\cdots>k_t\geq 0$, define $s_r=\sum_{j=1}^r 2^{k_j}$. We have:
\begin{eqnarray*}
T(n)=\sum_{j=0}^{n-1} H(j) &=& \sum_{j=0}^{s_1-1}H(j) + \sum_{j=s_1}^{s_2-1} H(j) + \cdots + \sum_{j=s_{t-1}}^{s_{t}-1} H(j)
\end{eqnarray*}
We now notice that for $s_{r-1} \leq j \leq s_{r}-1$, $H(j)=r-1+H(j-s_{r-1})$ so
$$
\sum_{j=s_{r-1}}^{s_{r}-1} H(j)=\sum_{j=s_{r-1}}^{s_{r}-1} (r-1+H(j-s_{r-1}))=\sum_{m=0}^{2^{k_{r}}-1} (r-1+H(m))=(r-1) 2^{k_{r}} + k_{r} 2^{k_{r}-1}
$$
Here we used that $T(2^k)=\sum_{m=0}^{2^k-1} H(m)=k2^{k-1}$ by part (a).

Then,
\begin{eqnarray*}
T(n) &=& \sum_{j=0}^{s_1-1}H(j) + \sum_{j=s_1}^{s_2-1} H(j) + \cdots + \sum_{j=s_{t-1}}^{s_t-1} H(j)\\
&=& k_1 2^{k_1-1} + 2^{k_2} + k_2 2^{k_2-1} + 2 2^{k_3} + k_3 2^{k_3-1} +\cdots + (t-1)2^{k_t} + k_t 2^{k_t-1}\\
&=& \sum_{j=1}^t (k_j 2^{k_j-1} +(j-1) 2^{k_j}).
\end{eqnarray*}

(c) Let $d$ be an integer. Note that $d=\ceil{\log n} \Leftrightarrow d-1 < \log n \leq d \Leftrightarrow 2^{d-1} < n \leq 2^{d}$. Then since $T(n)$ is clearly strictly monotonically increasing we have $(\ceil{\log n}-1) n/4\leq (d-1)2^{d-2}=T(2^{d-1})<T(n) \leq T(2^d)=d 2^{d-1} < n \ceil{\log n}$.

\end{proof}

\end{document}